\documentclass[11pt]{amsart}
\usepackage{amsfonts,latexsym,amsthm,amssymb,graphicx}
\usepackage[all]{xy}
\usepackage{hyperref}
\usepackage[usenames]{color} 

\DeclareFontFamily{OT1}{rsfs}{}
\DeclareFontShape{OT1}{rsfs}{n}{it}{<-> rsfs10}{}
\DeclareMathAlphabet{\mathscr}{OT1}{rsfs}{n}{it}

\newtheorem{theorem}{Theorem}[section]
\newtheorem{lemma}[theorem]{Lemma}
\newtheorem{corol}[theorem]{Corollary}
\newtheorem{prop}[theorem]{Proposition}
\newtheorem{claim}[theorem]{Claim}

{\theoremstyle{definition} \newtheorem{defin}[theorem]{Definition}}
{\theoremstyle{remark} \newtheorem{remark}[theorem]{Remark}
}

\numberwithin{equation}{section}

\newcommand{\Abb}{{\mathbb{A}}}
\newcommand{\Cbb}{{\mathbb{C}}}
\newcommand{\Pbb}{{\mathbb{P}}}
\newcommand{\Qbb}{{\mathbb{Q}}}
\newcommand{\Rbb}{{\mathbb{R}}}
\newcommand{\Zbb}{{\mathbb{Z}}}
\newcommand{\cL}{{\mathscr L}}
\newcommand{\cN}{{\mathscr N}}
\newcommand{\cO}{{\mathscr O}}
\newcommand{\ua}{\underline a}
\newcommand{\uDelta}{\underline\Delta}
\newcommand{\Til}[1]{{\widetilde{#1}}}
\newcommand{\id}{\text{id}}
\DeclareMathOperator{\Vol}{Vol}
\newcommand{\qede}{\hfill$\lrcorner$}

\title{
Newton-Okounkov bodies and Segre classes
}
\author{Paolo Aluffi}
\address{
Mathematics Department, 
Florida State University,
Tallahassee FL 32306, U.S.A.
}
\email{aluffi@math.fsu.edu}

\begin{document}

\begin{abstract}
Given a homogeneous ideal in a polynomial ring over $\Cbb$, we 
adapt the construction of Newton-Okounkov bodies to obtain a convex subset of Euclidean
space such that a suitable integral over this set computes the {\em Segre zeta function\/} 
of the ideal. That is, we extract the numerical information of
the Segre class of a subscheme of projective space from an associated (unbounded)
Newton-Okounkov convex set. The result generalizes to arbitrary subschemes of projective
space the numerical form of a previously known result for monomial schemes. 
\end{abstract}

\maketitle


\section{Introduction}\label{s:intro}
\subsection{}
Seminal work of R.~Lazarsfeld, M.~Musta{\c{t}}{\u{a}} (\cite{MR2571958}) and K.~Kaveh, 
A.~Khovanskii (\cite{MR2950767})  began the systematic study of {\em Newton-Okounkov
bodies\/} associated with (for instance) linear systems on a variety. One of the remarkable
features of the theory is a very compelling expression for the {\em intersection index\/} of 
a linear system. Roughly speaking, the intersection index of a linear system $L$ on an 
$n$-dimensional variety $V$ is the number of points of intersection of $n$ general elements 
of $L$, where one discards intersections occurring along the base locus of $L$. Kaveh and 
Khovanskii prove (\cite[Theorem~4.9]{MR2950767}) that, modulo important technicalities, 
this index equals the normalized volume of the corresponding Newton-Okounkov body. 
This theorem may be viewed as a vast generalization of the classical Kushnirenko theorem 
on the number of solutions in a torus of a system of general equations with given Newton 
polytope.

On the other hand, the intersection index of a linear system admits a transparent interpretation
in terms of standard intersection theory (this is observed in~\cite[\S7]{MR2722802}). The 
Fulton-MacPherson approach to intersection theory may then be used to express this index 
in terms of the {\em Segre class\/} (\cite[Chapter~4]{85k:14004}) of the base scheme of the 
system in any completion of~$V$. In fact, Segre classes are a considerably more refined type
of information: for example, arbitrary intersection products (not just intersection {\em numbers\/}) 
may be defined by means of Segre classes (\cite[Proposition~6.1(a)]{85k:14004}).
In view of the theorem of Kaveh and Khovanskii mentioned above, it is natural to ask whether
these more refined objects can be computed from a suitably constructed Newton-Okounkov
body.

In the case of the Kushnirenko theorem, this is indeed the case. The main result 
of~\cite{MR3576538} expresses the Segre class of a {\em monomial\/} ideal in terms of
a certain integral evaluated on the Newton polytope determined by the ideal. The result
yields (yet) another proof of Kushnirenko's theorem, generalizing it in a different
direction than the Kaveh-Khovanskii result. Thus, we have two generalizations of 
Kushnirenko's theorem:
\[
\xymatrix{
{\text{Kaveh-Khovanskii}} \ar@{=>}[r] & {\text{Kushnirenko}} \\
{\boxed{???}} \ar@{==>}[u] \ar@{==>}[r] & {\text{[Alu16]}} \ar@{=>}[u]
}
\]
The goal of this note is to fill in this diagram for arbitrary subschemes of projective space. 
Given a homogeneous ideal $I$ of $\Cbb[x_0,\dots, x_n]$, defining a subscheme $X$ of 
$\Pbb^n=\Pbb^n_\Cbb$, 
we construct an associated Newton-Okounkov body $\Delta(I)\subseteq \Rbb^{n+1}$ and 
prove that the push-forward of $s(X,\Pbb^n)$ to $\Pbb^n$ is evaluated by a suitable integral 
over this body. The result may be viewed as a common generalization 
of~\cite[Theorem~4.9]{MR2950767}
and of a `numerical' form of~\cite{MR3576538}, for subschemes of projective spaces.
(We note that \cite[Theorem~4.9]{MR2950767} is a key ingredient in the proof.) 
More precisely, the result is as follows.

\begin{theorem}\label{thm:segre1in}
Denote by $\iota$ the inclusion $X\hookrightarrow \Pbb^n$, and let $h$
be the hyperplane class in $\Pbb^n$. Define a power series $\sum \rho_i t^i$
by the identity
\[
\int_{\Delta(I)} \frac{(n+1)! t^{n+1} da_0\cdots da_n}
{(1+(a_0+\cdots + a_n)t)^{n+2}} = \sum_{i\ge 0} \rho_i t^i\quad.
\]
Then
\[
\iota_* s(X,\Pbb^n) = \left(1-\sum_{i=0}^n \rho_i h^i\right) \cap [\Pbb^n]\quad.
\]
\end{theorem}

The integral in this statement is computed formally, treating $t$ as a positive real parameter.
According to Theorem~\ref{thm:segre1in}, the first $n+1$ coefficients $\rho_0,\dots, \rho_n$ 
of its expansion in $t$ determine and are determined by the push-forward of of $s(X,\Pbb^n)$.

In fact, {\em all\/} coefficients $\rho_i$ admit a Segre class interpretation. We prove that
the integral appearing in Theorem~\ref{thm:segre1in} computes the {\em Segre zeta
function\/} $\zeta_I(t)$ of the ideal $I$. More precisely:

\begin{theorem}\label{thm:segre2in}
With notation as above,
\[
\int_{\Delta(I)} \frac{(n+1)! t^{n+1} da_0\cdots da_n}
{(1+(a_0+\cdots + a_n)t)^{n+2}} = 1- \zeta_I(t)\quad.
\]
\end{theorem}

The Segre zeta function $\zeta_I(t)$ is a power series evaluating the push-fowards of the
Segre classes $s(X^{(N)},\Pbb^N)$ of the cones $X^{(N)}$ over $X$ in $\Pbb^N$ for $N\ge n$.
In~\cite{MR3709134} it is shown that $\zeta_I(t)$ is a {\em rational\/} function; so must be
the integral appearing in Theorems~\ref{thm:segre1in} and~\ref{thm:segre2in}. In fact,
\[
\zeta_I(t) = \frac{A(t)+d_0\cdots d_r t^{r+1}}{(1+d_0 t)\cdots (1+d_r t)}
\]
where the integers $d_i$ are a subset of the degree sequence for $I$ and $A(t)$ is a 
polynomial of degree $\le r$ with nonnegative integer coefficients. It would be interesting to 
provide an interpretation or a different proof for these facts in terms of the Newton-Okounkov
body~$\Delta(I)$.\smallskip

The paper is organized as follows. In~\S\ref{s:prelim} we discuss preliminaries, including
short summaries of the definitions of Segre classes and of ordinary Newton-Okounkov
bodies, and of the Kaveh-Khovanskii intersection index of a linear system. We explain how the 
intersection index may be computed from the Segre class of an associated scheme
(Corollary~\ref{cor:intindsc}). We expand on this relation in the case of subschemes of 
projective space, and prove (Proposition~\ref{prop:IdS}) that knowledge of the intersection 
indices determined by linear systems associated with the graded pieces $I_s$ of a 
homogeneous ideal $I$, for $s\gg 0$, is in fact equivalent to knowledge of the push-forward
$\iota_* s(X,\Pbb^n)$ of the Segre class of the subscheme $X$ of $\Pbb^n$ defined by~$I$.
In~\S\ref{s:constr} we construct the Newton-Okounkov body $\Delta(I)$. The construction
is an adaptation of the `global Okounkov body' of~\cite[\S4]{MR2571958}, and depends on
the choice of a valuation on $\Cbb[x_1,\dots, x_n]$ (as well as on the choice of a
dehomogenizing parameter). The body $\Delta(I)$ is a closed convex 
cone in $\Rbb^{n+1}$; it maps to~$\Rbb^1$ by the function $(a_0,\dots, a_n)\mapsto \sum_i a_i$, 
and we prove (Proposition~\ref{prop:fibers}) that, for integer $s\gg 0$, the fiber over $s$ of this 
map equals the (conventional) Newton-Okounkov body of the graded piece $I_s$. This fact 
allows us to relate the volume of the fiber over $r\in \Rbb$, $r\gg 0$, to the degrees of the
components of $s(X,\Pbb^n)$. In~\S\ref{s:segre1} this relation is used to prove 
Theorem~\ref{thm:segre1in}. The argument is refined in~\S\ref{s:segre2} to yield the 
proof of the more precise Theorem~\ref{thm:segre2in}.\smallskip

If $I\subseteq \Cbb[x_0,\dots, x_n]$ is a {\em monomial\/} ideal in the indeterminates $x_i$, the
construction presented in~\S\ref{s:constr} simply reproduces the Newton polytope of $I$,
and Theorem~\ref{thm:segre1in} reproduces the corresponding consequence of the
result from~\cite{MR3576538}. We note that the result from~\cite{MR3576538} is
stronger, in
the sense that it is a statement about classes in the Chow group, rather than about their
degrees; further, it holds for monomial schemes based on a set of divisors meeting with
`regular crossings' (a substantial weakening of the normal crossing condition) on any
(not necessarily smooth) variety. We expect that a corresponding strengthening of 
Theorems~\ref{thm:segre1in} and~\ref{thm:segre2in} should hold, providing a computation 
of the Segre class
$s(X,V)$ as an integral on a Newton-Okounkov body generalizing the construction
given here, interpreted as a class in $A_*X$ in the same way as is done in the monomial 
case in~\cite{MR3576538}. A proper generalization should put little or no requirements on 
the ambient variety $V$. The proof of such a result would likely have to rely on different
techniques---we do not expect the `volume' considerations that lead to the proof of Theorems~\ref{thm:segre1in} and~\ref{thm:segre2in} in this paper to be adequate to deal with 
classes in the Chow group. A proof of the stronger result would likely ultimately rely on
the birational invariance of Segre classes and on the study of the effect of blow-ups on
a suitable generalization of the Newton-Okounkov body $\Delta(I)$ constructed here.

\subsection{Acknowledgments}
This work was carried out while the author was visiting the University of Toronto.
The author thanks the University of Toronto for the hospitality.


\newcommand{\T}{{T}}

\section{Preliminaries}\label{s:prelim}

\subsection{}
We work over $\Cbb$. This is necessary as the results from~\cite{MR2950767} we will use
in the proof are stated over $\Cbb$; we expect that the results of this paper should hold over
arbitrary algebraically closed fields.   

\subsection{}
Let $X\subsetneq V$ be a closed embedding of schemes; for convenience we assume $V$ to
be a variety. The {\em Segre class\/} $s(X,V)$
is an element of the Chow group $A_*X$ of $X$, characterized by the following properties:
\begin{itemize}
\item {\em Birational invariance:\/} If $f: V' \to V$ is a proper birational map, then 
\[
s(X,V) = f_* s(f^{-1}(X),V')\quad.
\]
\item {\em Divisors:\/} If $X$ is a Cartier divisor in $V$, then $s(X,V)=c(\cO(X))^{-1}\cap [X]$.
\end{itemize}
It is clear that these properties determine $s(X,V)$ for any closed subscheme $X\subsetneq V$:
by the first point we may blow-up $V$ along $X$, reducing to the case in which $f^{-1}(X)$ is
a divisor in $V$, and the second point determines the class in this case.

It can be shown that the second point generalizes to arbitrary regular embeddings: 
if $X\subsetneq V$ is a regular embedding, with normal bundle $\cN$, then 
$s(X,V)=c(\cN)^{-1} \cap [X]$. For a thorough treatment of Segre classes, the reader
is addressed to~\cite[Chapter 4]{85k:14004}.

The Fulton-MacPherson intersection product may be defined in terms of Segre classes.
If $X$ and $Y$ are subvarieties of a variety $V$, with $X$ regularly embedded in $V$
with normal bundle $\cN$, consider the fiber diagram
\[
\xymatrix{
X\cap Y \ar[r] \ar[d]_j & Y \ar[d] \\
X \ar[r] & V
}
\]
Then we may set
\[
X\cdot Y := \{c(j^*\cN)\cap s(X\cap Y,Y)\}_{\dim X+\dim Y-\dim V}
\]
where $\{\cdot \}_k$ stands for the component of dimension $k$ in the class within braces;
cf.~\cite[Proposition~6.1(a)]{85k:14004}.
This definition satisfies all expected properties of an intersection product, and may be used 
to give $A_*V$ the structure of a ring when $V$ is nonsingular (\cite{85k:14004}).

Thus, Segre classes may be viewed as a key tool in intersection theory, and this is our
motivation in seeking alternative ways to compute them. We also reproduce here the 
following result, which will be relevant to our discussion of the `intersection index' 
in~\S\ref{s:KKintind}. Let $L$ be a finite dimensional vector space of sections of a line
bundle $\cL$ on a compact $n$-dimensional variety $\overline V$. 
We have the associated `Kodaira' rational map $\varphi: \overline V\dashrightarrow \Pbb(L^\vee)$,  
and we let $\Til V\to \overline V$ be any resolution of the indeterminacies of $\varphi$, i.e., 
a proper birational morphism such that the composition with $\varphi$ determines a
{\em regular\/} map
\[
\xymatrix{
\tilde \varphi: \Til V \ar[r] & \Pbb(L^\vee)\quad.
}
\]
For example, $\Til V\to \overline V$ could be the blow-up of $\overline V$ along the base scheme 
$B$ of $L$ (i.e., the intersection of all divisors of $V$ corresponding to sections in $L$).

\begin{prop}\label{prop:44}
Let $H$ be the hyperplane class in $\Pbb(L^\vee)$. Then with notation as above
\[
\int (\tilde\varphi^* H)^n \cdot [\Til V] = \int c_1(\cL)^n\cap [\overline V] 
- \int c(\cL)^n\cap s(B,\overline V)\quad.
\]
\end{prop}

\begin{proof}
Every resolution $\Til V$ factors through the blow-up $B\ell_B\overline V$. Therefore we
may in fact assume that $\Til V$ is the blow-up, and then the statement is
\cite[Proposition 4.4]{85k:14004}.
\end{proof}

\subsection{}\label{ss:NOdef}
Let $V$ be an algebraic variety of dimension $n$ (not necessarily nonsingular or complete).
We can associate vector spaces of rational functions on $V$ with linear systems on any 
completion $\overline V$ of $V$: if $\cL$ is a line bundle and $L\subseteq H^0(\overline V, \cL)$
is a subspace, fix a nonzero section $s_0$ of $\cL$ and associate $s\in L$ with the
rational function $\frac s{s_0}$. The choice can be performed compatibly with products
(for example, we may choose $s_0^t$ to identify $L^t\subseteq H^0(\overline V, \cL^{\otimes t})$
with a space of rational functions).
This will be implicitly assumed in the following; 
we will abuse notation and use the same notation for a linear system $L$ and for a 
corresponding space of rational functions.

We fix a $\Zbb^n$-valued valuation $v$ on the field of rational 
functions on $V$. For example, $v$~could be the valuation associated with an `admissible 
flag' as in~\cite[\S1]{MR2571958}; for a general discussion of valuations in the context
needed here, see~\cite[\S2.2]{MR2950767}.

For $t\in \Zbb^{\ge 0}$, the spaces $L^t$ (spanned by products of $t$-tuples of elements 
of $L$) determine subsets $v(L^t):=v(L^t\smallsetminus 0)$ 
of $\Zbb^n$. By construction, $v(L^t)\subseteq v(L^u)$ for $t\le u$. 
The Newton-Okounkov body $\Delta(L)$ captures the asymptotic behavior of $v(L^t)$ as 
$t\to \infty$. To construct~$\Delta(L)$, consider the set
\[
U = \{(\ua, u)\in \Rbb^n\times \Rbb^1\,|\, u\in \Zbb^{\ge 0}, \ua\in v(L^u)\}\quad;
\]
$U$ is the {\em graded semigroup\/} in the terminology of~\cite{MR2571958}. Let 
$\Sigma(U)$ be the closed convex cone spanned by $U$. The Newton-Okounkov
body of $L$ is obtained by setting the last coordinate $u$ to $1$ in $\Sigma(U)$.

\begin{defin}[\cite{MR2571958}, Definition~1.8; \cite{MR2950767}, \S1]
The {\em Newton-Okounkov body\/} of $L$ is the convex set
\[
\Delta(L) := \Sigma(U)\cap (\Rbb^n\times \{1\})\quad,
\]
viewed as a subset of $\Rbb^n$.
\qede\end{defin}

The Newton-Okounkov body of a space $L$ is a closed, convex, and compact
subset of $\Rbb^n$. Its definition depends on the valuation $v$ and, for linear systems,
on the chosen identification with spaces of rational functions. These choices will be
inessential in what follows, so they are omitted from the notation.

\subsection{}\label{s:KKintind}
Kaveh and Khovanskii associate an 
{\em intersection index,\/} denoted $[L,\dots, L]$ in~\cite{MR2950767} 
with every space $L$ of rational functions as above, and more generally with any choice of
$n$ spaces $L_1,\dots, L_n$. The index $[L,\dots, L]$ equals the number of 
solutions in $V$ of a system of equations 
$\ell_1 = \cdots = \ell_n = 0$, where each $\ell_i$ is a general element in $L$, and one neglects
intersections at which all $\ell\in L$ vanish, and those occurring where some of the 
functions $\ell_i$
have poles. By construction, the intersection index of $L$ does not change if we replace 
$V$ by a dense open subset, so we may in fact assume that all functions in~$L$ are regular.
If the space arises from a linear system $L$ as in~\S\ref{ss:NOdef}, the index is clearly
independent of the chosen identification; we will therefore use the notation $[L,\dots, L]$
in this case as well.

We will first recall an intersection-theoretic interpretation of the Kaveh-Khovanskii intersection
index; this is a minimal adaptation of the treatment in~\cite[\S7]{MR2722802}.
Any choice $f_0,\dots, f_r$ of generators of $L$ determines a rational map
\[
\xymatrix{
\varphi: V \ar@{-->}[r] &  \Pbb^r\quad,
}
\] 
mapping $p\in V$ to $(f_0(p)\colon \dots \colon f_r(p))$ provided $f_i(p)\ne 0$ for some $p$.
Note that there is a canonical injection of $\Pbb(L^\vee)$ into $\Pbb^r$, and $\varphi$
factors
\[
\xymatrix{
\varphi:\quad V \ar@{-->}[r] & \Pbb(L^\vee) \ar@{^(->}[r] & \Pbb^r
}
\]
where $V \dashrightarrow \Pbb(L^\vee)$ is the standard Kodaira rational map.
Extend $\varphi$ to any completion $\overline V$ of $V$, and let $\pi: \Til V\to \overline V$ 
be any birational map resolving the indeterminacies of $\varphi$:
\[
\xymatrix{
& \Til V \ar[ld]_\pi \ar[rd]^{\tilde \varphi} \\
\overline V \ar@{-->}[rr]^\varphi & & \Pbb^r
}
\]
For example, $\Til V$ could be the closure of the graph of $\varphi$, or equivalently
the blow-up of $\overline V$ along the base scheme of $L$.
Finally, let $H$ be the hyperplane class in $\Pbb^r$. The following is essentially
a particular case of~\cite[Corollary~7.7]{MR2722802}.

\begin{lemma}\label{lem:kk}
With notation as above, 
\[
[L,\dots, L] = \int (\tilde\varphi^* H)^n \cdot [\Til V]\quad.
\]
\end{lemma}

\begin{proof}
Since $V$ and $\Til V$ share a dense open subset, the index $[L,\dots, L]$ may be 
computed on~$\Til V$. There $L$ corresponds to a base point free linear system, whose
elements are preimages of hyperplanes from $\Pbb^r$. The index then agrees with
the ordinary intersection product, and this is the statement.
\end{proof}

In fact, identify $L$ with a linear system of sections of a line bundle $\cL$ on $\overline V$, 
and let $B$ be its base locus, i.e., the intersection of all divisors defined by nonzero sections 
in the system.

\begin{corol}\label{cor:intindsc}
With notation as above,
\begin{equation}\label{eq:intindsc}
[L,\dots, L] = \int c_1(\cL)^n\cap [\overline V] - \int c(\cL)^n\cap s(B,\overline V)
\end{equation}
\end{corol}

\begin{proof}
This follows from Lemma~\ref{lem:kk} and Proposition~\ref{prop:44}.
\end{proof}

\begin{remark}
We could adopt~\eqref{eq:intindsc} as the {\em definition\/} of the intersection index
$[L,\dots, L]$, and use it to give a treatment of the index over more general fields.
In the following, we will exploit a relation between the intersection index and Segre classes
of which~\eqref{eq:intindsc} is the most straightforward manifestation; see
especially Proposition~\ref{prop:IdS} and Corollary~\ref{cor:frac}.
\qede\end{remark}

\subsection{}
Composing the morphism $\varphi: V \dashrightarrow \Pbb^r$ with an $a$-Veronese
embedding, we see that
\begin{equation}\label{eq:multiadd}
[L^a,\dots, L^a] = \int (\tilde\varphi^* (aH))^n \cdot [\Til V]
=a^b \int (\tilde\varphi^* H)^n \cdot [\Til V] = a^n [L,\dots, L]
\end{equation}
by Lemma~\ref{lem:kk}. And indeed, the Kaveh-Khovanskii intersection index is `multiadditive' (\cite[Theorem~4.7(1)]{MR2950767}). 

If $L,M$ are two nonzero finite dimensional subspaces of rational functions on a 
variety~$V$, choices of generators $f_0,\dots, f_r$ for $L$ and $g_0,\dots g_s$ for $M$
determine a set of generators $f_i g_j$ for $LM$. The corresponding rational map
$\psi: V \dashrightarrow \Pbb^{rs+r+s}$ factors through the Segre embedding:
\[
\xymatrix{
\psi\colon\quad V \ar@{-->}[r]^-\varphi &  \Pbb^r\times \Pbb^s \ar@{^(->}[r] & \Pbb^{rs+r+s}\quad.
}
\] 
The hyperplane class in $\Pbb^{rs+r+s}$ pulls back to the sum $h+k$ of the (pull-backs
of the) hyperplane classes in the two factors. By Lemma~\ref{lem:kk}, 
\[
[LM,\dots, LM] =  \int (\tilde\varphi^* (h+k))^n \cdot [\Til V]
=\sum_{i=0}^n \binom ni  \int (\tilde\varphi^* h^i k^{n-i}) \cdot [\Til V]
\quad,
\]
where again $\tilde\varphi : \Til V \to V$ resolves the indeterminacies of $\varphi$.
It follows easily that
\[
[\underbrace{L,\dots, L}_i, \underbrace{M,\dots, M}_{n-i}] =
\int (\tilde\varphi^* h^i k^{n-i}) \cdot [\Til V]\quad.
\]
The same technique may be used to express $[L_1,\dots, L_n]$ as an ordinary intersection 
product for any choice of $n$ finite dimensional vector spaces of rational functions, 
generalizing Lemma~\ref{lem:kk}. (See e.g., \cite[Corollary~7.7]{MR2722802}.)
Here we note the following observation, for later use. 

\begin{lemma}\label{lem:inj}
With notation as above, assume that the Kodaira rational map associated with~$M$ is
generically injective. Then the Kodaira rational map associated with $LM$ is generically
injective. 
\end{lemma}

\begin{proof}
Still using notation as above, the morphism $V \dashrightarrow \Pbb^s$ is generically injective 
by hypothesis, therefore so is $\varphi: V \dashrightarrow \Pbb^r\times \Pbb^s$. It follows
that $\psi: V \dashrightarrow \Pbb^{rs+r+s}$ is generically injective, and then so must be
the corresponding Kodaira rational map, since $\psi$ factors through it.
\end{proof}

\subsection{}
Now let $I\subseteq \Cbb[x_0,\dots, x_n]$ be a homogeneous ideal: 
$I=\oplus_{s\ge 0}I_s$, where $I_s$ is the piece in degree $s$. 
The {\em degree sequence\/} of $I$ is the list $(d_0,\dots,d_r)$ of degrees of a minimal
set of homogeneous generators for $I$; this depends only on $I$. In particular, so does the
largest element $d:=d_r$; we will call this the {\em generating degree\/} for $I$. Note that
$I_d\cdot \Cbb[x_0,\dots, x_n]=\oplus_{s\ge d} I_s$. We record the following elementary fact.

\begin{lemma}\label{lem:ILM}
Let $L=I_d$, where $d$ is the generating degree of $I$, and let $M=\langle x_0,\dots, x_n\rangle$.
Then for $a,b\in \Zbb$ such that $a\ge 1$ and $b\ge d a$ we have
\begin{equation}\label{eq:ILM}
(I^a)_b = L^a M^{b-d a}\quad.
\end{equation}
\end{lemma}

For example, it follows that for every $c\ge 0$, 
\begin{equation}\label{eq:pow}
((I^a)_b)^c = (L^a M^{b-d a})^c = L^{ac} M^{(b-da)c} = (I^{ac})_{bc}
\end{equation}
if $a\ge 1$ and $b\ge da$.

Next, let $\iota: X\hookrightarrow \Pbb^n$ be the closed subscheme defined by $I$ in $\Pbb^n$; 
we assume $X\subsetneq \Pbb^n$. 
Our task here is to relate the push-forward of the Segre class of $X$ to $\Pbb^n$ to the 
asymptotic behavior of the graded pieces $I_s$ of $I$, in terms of associated Kaveh-Khovanskii
intersection indices.
We will view each $I_s$ as a linear system, determining a rational map
\begin{equation}
\varphi_s:\Pbb^n \dashrightarrow \Pbb^{N_s}:=\Pbb({I_s}^\vee)\quad.
\end{equation}
The indeterminacies of $\varphi_s$ are resolved
by the closure $\Gamma_s$ of the graph of $\varphi_s$, a subvariety of dimension $n$ 
of $\Pbb^n \times \Pbb^{N_s}$.
The class of $\Gamma_s$ in the Chow group $A_n(\Pbb^n \times \Pbb^{N_s})$ may be 
written as
\[
[\Gamma_s] = g_0^{(s)} H^{N_s} + \cdots + g_n^{(s)} h^n H^{N_s-n}\quad,
\]
for integers $g_0^{(s)},\dots, g_n^{(s)}$. Here, $h$ and $H$ denote 
the (pull-backs of the) hyperplane classes from the first and second factor. We let
\[
G_s:= g_0^{(s)} + g_1^{(s)} h \cdots + g_n^{(s)} h^n 
\]
be the `shadow' of $[\Gamma_s]$ in $\Pbb^n$. The integers $g_i^{(s)}$ are the
{\em multidegrees\/} of the rational map~$\varphi_s$; the class $G_s$ packages 
the information of the multidegrees into a single class in $A_*\Pbb^n$.

It follows easily from the definition that $g^{(s)}_0=1$.
By definition, the `top' multidegree $g^{(s)}_n$ equals
\begin{equation}\label{eqn:top}
g^{(s)}_n = H^n \cdot [\Gamma_s]\quad;
\end{equation}
thus, $g^{(s)}$ equals the degree of the closure of the image of $\varphi_s$ 
if $\varphi_s$ is generically injective, hence birational onto its image. 
By~\eqref{eqn:top} and Lemma~\ref{lem:kk} 
we have
\begin{equation}\label{eqn:KK}
g^{(s)}_n = [I_s,\dots, I_s]\quad,
\end{equation}
the Kaveh-Khovanskii intersection index of $I_s$.

\subsection{}
We will use notation as in~\cite[\S2]{MR96d:14004}: for $G=\sum g_i h^i \in A_*\Pbb^n$
and $a\in \Zbb$, we let
\begin{equation}\label{eq:not}
G\otimes \cO(ah) := \sum_{i=0}^n \frac {g_i h^i}{(1+ah)^i}\quad.
\end{equation}
In fact, we consider the class
\[
\T_a(G):=\frac 1{1-ah} \left( G\otimes \cO(-ah)\right)\quad;
\]
explicitly,
\begin{equation}\label{eq:expli}
\T_a(G) = \sum_{i\ge 0} \left(\sum_{j=0}^i \binom ij a^{i-j} g_j\right) h^i\quad.
\end{equation}

\begin{lemma}\label{lem:shift}
With notation as above:
\begin{itemize}
\item[(i)] For $a,b\in \Zbb$, $\T_a(\T_b(G))=\T_{a+b}(G)$.
\item[(ii)] For $s\gg 0$, $\T_{-s}(G_s)$ is independent of $s$.
\item[(iii)] In fact, $\T_{-s}(G_s)=[\Pbb^n]-\iota_* s(X,\Pbb^n)$ for $s \ge d$, the generating 
degree of $I$.
\end{itemize}
\end{lemma}

\begin{proof}
(i): The notation~\eqref{eq:not} defines an action (\cite[Proposition~2]{MR96d:14004}),
therefore (using~\cite[Proposition~1]{MR96d:14004})
\begin{align*}
\T_a(\T_b(G)) &= \frac 1{1-ah}\left(\left(\frac 1{1-bh} (G\otimes \cO(-bh))\right)
\otimes \cO(-ah)\right)\\
&=\frac 1{1-ah}\left(\frac {1-ah}{1-(a+b)h} (G\otimes \cO(-bh) \otimes \cO(-ah))\right) \\
&=\frac 1{1-(a+b)h} \left(G\otimes \cO(-(a+b)h)\right) \\
&=\T_{a+b}(G)\quad.
\end{align*}

(ii) follows from (iii). To prove (iii), let $s\ge d$. Then $X$ is defined scheme-theoretically by
$I_s$, and (iii) follows from~\cite[Proposition~3.1]{MR1956868}.
\end{proof}

\subsection{}
For notational convenience, define integers $\sigma_j$, $j=0,\dots, n$ for $X\subsetneq \Pbb^n$
as above, so that
\begin{equation}\label{eq:defsig}
[\Pbb^n]-\iota_* s(X,\Pbb^n) = \sum_{j=0}^n \sigma_j [\Pbb^{n-j}]\quad:
\end{equation}
that is, $\sigma_0=1$ while $-\sigma_j$ is the degree of the component of $s(X,\Pbb^n)$
of codimension $j$ in~$\Pbb^n$, for $j>0$.

\begin{prop}\label{prop:IdS}
Let $d$ be the generating degree for $I$, and let $s\ge d$. Then
\begin{equation}\label{eq:intindse}
 [I_s,\dots, I_s] = \sum_{j=0}^n \binom nj \sigma_j s^{n-j} 
\end{equation}
\end{prop}

\begin{proof}
By Lemma~\ref{lem:shift}(i), $\T_s\circ \T_{-s}=\id$. Therefore, Lemma~\ref{lem:shift}(iii)
gives
\[
G_s=\T_s\circ \T_{-s}(G_s)=\T_s\big([\Pbb^n]-\iota_* s(X,\Pbb^n)\big)
\]
for $s\ge d$. By~\eqref{eqn:KK}, $[I_s,\dots, I_s]$ equals the
coefficient of $h^n$ in $G_s$, and this gives the statement by~\eqref{eq:expli}. 
\end{proof}

As a consequence of Proposition~\ref{prop:IdS}, the information carried by $\iota_* s(X,\Pbb^n)$ 
is equivalent to the information carried by the intersection indices 
$[I_s,\dots, I_s]$ for $s\gg 0$. In one form or another, this observation is 
at the root of most methods used for the algorithmic computation of Segre classes, 
starting with~\cite{MR1956868}.

Our task is to extract $\iota_* s(X,\Pbb^n)$ from a suitably constructed Newton-Okounkov body.
For this purpose, we will need to generalize~\eqref{eq:intindse} to rational $s$.
Let $q=\frac ba\in \Qbb$, with $a,b\in \Zbb^{>0}$, and assume $q\ge d$. It will be convenient 
to adopt the following notation:
\begin{equation}\label{eq:Iq}
[I_q,\dots, I_q] 
:= \frac 1{a^n} [(I^a)_b,\dots, (I^a)_b]\quad.
\end{equation}

Note that $[I_q,\dots, I_q]$ is well-defined for rational $q\ge d$.
Indeed, for $a,b\in \Zbb$ such that $a\ge 1$ and $b\ge d a$ we have $(I^{ac})_{bc} = ((I^a)_b)^c$
for every $c\ge 0$, by~\eqref{eq:pow}; by \eqref{eq:multiadd}, i.e., the multiadditivity of the 
intersection index,
\[
[(I^{ac})_{bc},\dots, (I^{ac})_{bc}] 
= [((I^a)_b)^c,\dots,((I^a)_b)^c] 
= c^n [(I^a)_b,\dots,(I^a)_b]\quad,
\]
and therefore
\[
\frac 1{(ac)^n}[(I^{ac})_{bc},\dots, (I^{ac})_{bc}]  = \frac 1{a^n}  [(I^a)_b,\dots,(I^a)_b]
\]
as needed.

The formula obtained in Proposition~\ref{prop:IdS} for the integral intersection index
remains true for the fractional version.

\begin{corol}\label{cor:frac}
Let $d$ be the generating degree for $I$, and let $q\in \Qbb$, $q\ge d$. Then with notation as
above
\[
[I_q,\dots, I_q] = \sum_{j=0}^n \binom nj \sigma_j q^{n-j} \quad.
\]
\end{corol}

\begin{proof}
Let $\iota^{(a)}: X^{(a)}\hookrightarrow \Pbb^n$ be the subscheme defined by the ideal $I^a$. 
Then $s(X^{(a)}, \Pbb^n)$ is obtained from $s(X,\Pbb^n)$ by multiplying by $a^i$ the component
of codimension $i$ in $\Pbb^n$ (\cite{MR96d:14004}). Therefore
\[
[\Pbb^n]-\iota^{(a)}_* s(X^{(a)},\Pbb^n) = \sum_{j=0}^n \sigma_j a^j [\Pbb^{n-j}]\quad.
\]
Further, note that the generating degree $I^a$ is $ad$. 
By Proposition~\ref{prop:IdS}, for $b\in \Zbb$, $b\ge ad$,
\[
[(I^a)_b,\dots, (I^a)_b] = \sum_{j=0}^n \binom nj \sigma_ja^j b^{n-j}\quad.
\]
With $q=\frac ba\ge d$, it follows that
\[
[I_q,\dots, I_q] = \frac 1{a^n} [(I^a)_b,\dots, (I^a)_b] 
= \frac 1{a^n} \sum_{j=0}^n \binom nj \sigma_j a^jb^{n-j} 
= \sum_{j=0}^n \binom nj \sigma_j\frac{b^{n-j}}{a^{n-j}} \quad,
\]
and this is the statement.
\end{proof}


\section{Construction}\label{s:constr}

\subsection{}\label{ss:cprels}
In this section we provide our construction of the Newton-Okounkov body associated with
a homogeneous ideal $I\subseteq \Cbb[x_0,\dots, x_n]$.
The construction is an adaptation of the construction of the `global Okounkov body' 
given in~\cite[\S4]{MR2571958}; as in {\it loc.~cit.} we abuse language and refer to
the result of the construction as a {\em body,\/} although the object is not compact.
It will be a closed convex subset of $\Rbb^{n+1}$ with nonempty interior.

As in~\S\ref{s:prelim}, we consider the homogeneous pieces of $I$ and of its powers.
The body will map to~$\Rbb^1$ in such a way that the fiber over $s\in \Zbb^{\gg 0}$ will
be the (ordinary) Newton-Okounkov body associated with $I_s$. As in~\S\ref{ss:NOdef},
we use a section of the line bundle $\cO(s)$, chosen compatibly with products,
to identify the linear systems ${I^t}_s\subseteq H^0(\Pbb^n,\cO(s))$ with spaces of rational 
functions on $\Pbb^n$.
For example, we can choose $x_0^s$; equivalently, we can de-homogenize by setting
$x_0=1$. This will be done implicitly in what follows.

Fix a $\Zbb^n$-valued valuation $v$ on $\Cbb[x_1,\dots, x_n]$. 
Consider the subset of $\Rbb^n\times \Rbb^2$ defined by
\begin{equation}\label{eq:dUI}
U_I:=\{(\ua, s,t)\in \Rbb^n\times \Rbb^2 \,|\, s\in \Zbb^{\ge 0}, t\in \Zbb^{\ge 0}, 
\ua\in v((I^t)_s)\}\quad.
\end{equation}
Let $\Sigma(U_I)$ be the closed convex cone generated by $U_I$. Setting the last coordinate 
to~$1$ defines a hyperplane $\{t=1\}\cong \Rbb^n\times \Rbb^1$. The intersection
\begin{equation}\label{eq:uDdef}
\uDelta(I):= \Sigma(U_I)\cap \{t=1\}\subseteq \Rbb^n\times \Rbb^1
\end{equation}
is a closed convex set. 
The projection $\Rbb^n\times \Rbb^1\to \Rbb^1$, $(\ua,s)\mapsto s$, defines
a projection $\pi: \uDelta(I)\to \Rbb^1$. We denote by $\uDelta_s$ the fiber $\pi^{-1}(s)$,
viewed as a subset of $\Rbb^n$.

\begin{defin}\label{def:NO}
The {\em Newton-Okounkov body of the ideal $I$\/} is the image $\Delta(I)$ 
of $\uDelta(I)$ via the isomorphism $\tau:\Rbb^n\times \Rbb^1 \to \Rbb^{n+1}$,
$((a_1,\dots, a_n), s) \mapsto (s-(a_1+\cdots+a_n), a_1,\dots, a_n)$.
\qede\end{defin}

\begin{remark}\label{rem:nei}
In the proof of the proposition that follows it will be shown that $\Delta(I)$ has
non-empty interior; in this sense it is a `body'.
\qede\end{remark}

\begin{remark}\label{rem:alt}
An alternative description of the body may be given in the style of~\cite[(1.5)]{MR2571958}:
\[
\Delta(I) = \text{closed convex hull } \left(\bigcup_{t'\ge 1} \frac 1{t'}\cdot \tau(U_{I,{t'}})\right)\quad,
\]
where $U_{I,{t'}} = U_I \cap \{t=t'\}$.
\qede\end{remark}

\subsection{}
Recall that $d$ denotes the generating degree for $I$.

\begin{prop}\label{prop:fibers}
Let $q\in \Qbb$, $q>d$, and write  $q=\frac ba$ with $a,b\in \Zbb^{>0}$.\newline
Then $\uDelta_q=\frac 1a \Delta((I^a)_b)$.
\end{prop}

\begin{proof}
First, assume that $a$ and $b$ are relatively prime. Let 
\[
U_{a,b} = \{(\ua,u)\in \Rbb^n\times \Rbb^1\,|\, u\in \Zbb^{\ge 0}\,,\, \ua\in v(((I^a)_b)^u)\}\quad;
\]
so $\Delta((I^a)_b)=\Sigma(U_{a,b})\cap \{u=1\}$ by definition. 
By~\eqref{eq:pow} we have $((I^a)_b)^u = (I^{au})_{bu}$ for $u\ge 0$ (since $b\ge da$), hence
\[
U_{a,b} = \{(\ua,u)\in \Rbb^n\times \Rbb^1\,|\, u\in \Zbb^{\ge 0}\,,\, \ua\in v((I^{au})_{bu})\}\quad.
\]
On the other hand, consider the hyperplane $H_q\subseteq \Rbb^n\times \Rbb^2$ 
consisting of $(\ua,s,t))$ with $s=qt$. Since $(s,t)\in \Zbb^2$ for $(\ua,s,t)$ in $U_I$ and 
$a$, $b$ are relatively prime, along $U_I$ the condition $s=qt$ is equivalent to 
$(s,t) = (bu,au)$ for some $u\in \Zbb$. It follows that the linear map $(\ua,u)\mapsto (\ua, bu, au)$
induces a bijection
\begin{equation}\label{eq:Uab}
\xymatrix{
U_{a,b} \ar[r]^-\cong & U_I\cap H_q\quad,
}
\end{equation}
which we will use to identify these two sets. (And note $t=au$ under this identification.)

Since $\frac ba>d$, $(I^a)_b = L^a M^{b-da}$ with $L=I_d$, $M=\langle x_0,\dots, x_n\rangle$
(Lemma~\ref{lem:ILM}).
Therefore the Kodaira map associated with $(I^a)_b$ is generically injective by Lemma~\ref{lem:inj}.
It follows that $\Delta((I^a)_b)$ and hence $\Sigma(U_{a,b})$ are full-dimensional
(since $\Delta((I^a)_b)$ has nonzero volume, 
cf.~\cite[Proposition~4.8 and Theorem~4.9(a)]{MR2950767}): 
i.e., $\Sigma(U_{a,b})$ has dimension $n+1$. 
By~\eqref{eq:Uab}, $\Sigma(U_I) \cap H_q$ {\em contains\/} (a copy of) $\Sigma(U_{a,b})$. 
Therefore, hyperplane sections of $\Sigma(U_I)$ have dimension $n+1$, and it follows that
$\Sigma(U_I)$ has dimension $n+2$, that is, it has nonempty interior. 
This also proves that $\uDelta(I)$ and hence $\Delta(I)$ have nonempty interior, as
promised in Remark~\ref{rem:nei}. 

Further, the argument shows that the hyperplane $H_q$ meets the interior of $\Sigma(U_I)$.
The hypotheses of~\cite[Proposition~A.1]{MR2571958} are satisfied, therefore
\[
\Sigma(U_I) \cap H_q = \Sigma(U_I \cap H_q) = \Sigma(U_{a,b})\quad,
\]
and hence
\[
\uDelta_q = \big(\Sigma(U_I) \cap \{t=1\}\big) \cap \{s=q\} 
= \Sigma(U_I) \cap H_q \cap \{t=1\} 
= \Sigma(U_{a,b}) \cap \{au=1\}\quad.
\]
Since $\Sigma(U_{a,b})$ is a cone,
\[
\Sigma(U_{a,b}) \cap \{au=1\} = \frac 1a \left(\Sigma(U_{a,b}) \cap \{u=1\}\right)
=\frac 1a \Delta((I^a)_b)\quad,
\]
concluding the verification if $a$ and $b$ are relatively prime. The general case follows,
since $\frac 1{ac} \Delta((I^{ac})_{bc}) = \frac 1a \Delta((I^a)_b)$ for all $c>0$, again since
$\Sigma(U_{a,b})$ is a cone.
\end{proof}

\subsection{}
Recall the definition of the integers $\sigma_j$, from~\eqref{eq:defsig}. In the 
following statement, $\Vol$ stands for the `normalized' $n$-dimensional volume, 
that is, $n!$ times the ordinary Euclidean volume in dimension~$n$
(denoted $\Vol_n$ in~\cite{MR2950767}).

\begin{corol}\label{cor:vols}
Let $r\in \Rbb$, $r>d$. Then $\Vol(\uDelta_r) = \sum_{i=0}^n \binom ni \sigma_{n-i} r^i$.
\end{corol}

\begin{proof}
By continuity, it suffices to verify the given formula for $r=q\in \Qbb$, $q>d$. 
Let then $q\in \Qbb$, $q=\frac ba>d$, with $a,b$ positive integers. 

Let $L=I_d$, $M=\langle x_0,\dots, x_n\rangle$. By Lemma~\ref{lem:ILM} we have 
$(I^a)_b=L^a M^{b-da}$; note that $b-da>0$ since $q>d$. By Lemma~\ref{lem:inj}, the
Kodaira map associated with $(I^a)_b$ is birational onto its image. 
By~\cite[Theorem~4.9(2)]{MR2950767},
\[
[(I^a)_b,\dots, (I^a)_b] = \Vol(\Delta((I^a)_b))
\]
(we are normalizing the volume by the factorial of the dimension), and therefore
\[
\Vol\left(\frac 1a \Delta((I^a)_b)\right) = \frac 1{a^n} \Vol(\Delta((I^a)_b)) 
=\frac 1{a^n} [(I^a)_b,\dots, (I^a)_b] 
= [I_q,\dots, I_q]
\]
adopting~\eqref{eq:Iq}. By Proposition~\ref{prop:fibers} and Corollary~\ref{cor:frac}, 
it follows that
\[
\Vol(\uDelta_q) = \sum_{i=0}^n \binom ni \sigma_{n-i} q^i
\]
as needed.
\end{proof}

\subsection{}\label{ss:monom}
If $I$ is generated by {\em monomials\/} in the variables $x_0,\dots, x_n$, then the construction presented above, with respect to the monomial valuation determined by any $n$ of the variables,
reduces to the usual Newton polytope.

Indeed, for a polynomial $P\in \Cbb[x_0,\dots, x_n]$, we may 
let $v(P)=(a_1,\dots, a_n)$ be the smallest exponent list of a monomial in $P|_{x_0=1}$, 
with respect to the lexicographic order; for example, $v(x_0^{a_0}x_1^{a_1}\cdots x_n^{a_n}) = 
(a_1,\dots, a_n)$.
This leads easily to an effective description of the sets $v((I^t)_s)$:
\[
(a_1,\dots, a_n) \in v((I^t)_s) \iff  x_0^{s-a_1-\dots -a_n} x_1^{a_1}\cdots x_n^{a_n} \in (I^t)_s\quad.
\]
It follows that
\[
U_I:=\{(\ua, s,t) \,|\, s\in \Zbb^{\ge 0}, t\in \Zbb^{\ge 0}, \tau(\ua,s)\in m(I^t)\}
\]
where $m(I^t)\subseteq \Zbb^{n+1}$ is the set of exponents of monomials in $I^t$, and
$\tau$ is the isomorphism defined in Definition~\ref{def:NO}.
For a fixed $t=t'$, and with $U_{I,t'} = U_I \cap \{t=t'\}$, this shows that
\[
\tau(U_{I,t'}) = m(I^{t'})\quad.
\]
The alternative description of $\Delta(I)$ given in Remark~\ref{rem:alt} then gives
\[
\Delta(I) = \text{convex hull }\big( m(I) \big)\subseteq \Rbb^{n+1}\quad,
\]
and this is the (unbounded) Newton polytope associated with $I$.


\section{From the Newton-Okounkov body to the Segre class}\label{s:segre1}

Let $X\subseteq \Pbb^n$ be a closed subscheme, and let $I\subseteq \Cbb[x_0,\dots, x_n]$
be any homogeneous ideal defining $X$ scheme-theoretically.

We have associated with $I$ a `Newton-Okounkov body' (Definition~\ref{def:NO})
$\Delta(I)$. This body depends on the defining ideal $I$, on the dehomogenizing factor,
and the choice of a
valuation. The following result expresses the degrees of the components of the
Segre class of $X$ in $\Pbb^n$ in terms of the truncation of a series computed
as an integral over $\Delta(I)$. 

By definition, the body $\Delta(I)$ is a subset of $\Rbb^{n+1}$. We let
$a_0,\dots, a_n$ denote coordinates in this space. The integral appearing
in the result is the ordinary `calculus' integral, depending on a parameter $t$
(which we may take to range in $\Rbb^{>0}$).

\begin{theorem}\label{thm:segre1}
Denote by $\iota$ the inclusion $X\hookrightarrow \Pbb^n$, and let $h$
be the hyperplane class in $\Pbb^n$. Define a power series $\sum \rho_i t^i$
by the identity
\begin{equation}\label{eq:int1}
\int_{\Delta(I)} \frac{(n+1)! t^{n+1} da_0\cdots da_n}
{(1+(a_0+\cdots + a_n)t)^{n+2}} = \sum_{i\ge 0} \rho_i t^i\quad.
\end{equation}
Then
\[
\iota_* s(X,\Pbb^n) = \left(1-\sum_{i=0}^n \rho_i h^i\right) \cap [\Pbb^n]\quad.
\]
\end{theorem}

In other words, the coefficients $\rho_0,\dots, \rho_n$ of the series defined 
in~\eqref{eq:int1} agree with the numbers $\sigma_0,\dots, \sigma_n$ defined 
in~\eqref{eq:defsig}.

\begin{proof}
We perform a change of variables, using the isomorphism~$\tau$ from
Definition~\ref{def:NO}:
\[
\tau((a_1,\dots, a_n), s) = (s-(a_1+\cdots+a_n), a_1,\dots, a_n)\quad.
\]
The absolute value of the jacobian is~$1$, therefore
\begin{align*}
\int_{\Delta(I)} \frac{(n+1)! t^{n+1} da_0\cdots da_n}
{(1+(a_0+\cdots + a_n)t)^{n+2}} 
&=\int_{\uDelta(I)} \frac{(n+1)! t^{n+1} da_1\cdots da_{n-1}ds}
{(1+(s-(a_1+\cdots+a_n)+ a_1+\cdots + a_n)t)^{n+2}} \\
&=\int_{\uDelta(I)} \frac{(n+1)! t^{n+1} da_1\cdots da_{n-1}ds}
{(1+st)^{n+2}} \quad.
\end{align*}
We have $s\ge 0$ on $\uDelta(I)$, so this integral equals
\[
\int_0^\infty \left(\int_{\uDelta_s} da_1\cdots da_n\right) \frac{(n+1)!\, t^{n+1}}
{(1+st)^{n+2}}\, ds
=(n+1) \int_0^\infty \Vol(\uDelta_s) \frac{t^{n+1}}{(1+st)^{n+2}}\, ds\quad.
\]
Now recall that $d$ denotes the generating degree for $I$. The function $\Vol(\uDelta_s)$
is bounded (and independent of $t$) on $s\in [0,d]$; therefore
\[
\int_0^d \Vol(\uDelta_s) \frac{t^{n+1}}{(1+st)^{n+2}}\, ds \equiv 0 \mod (t^{n+1})
\]
as a series in $t$. It follows that
\[
(n+1) \int_d^\infty \Vol(\uDelta_s) \frac{t^{n+1}}{(1+st)^{n+2}}\, ds
=\rho_0+\rho_1 t+ \cdots +\rho_n t^n + \text{h.o.t.}\quad:
\]
that is, since we are only interested in the coefficients $\rho_i$ with $i\le n$, we can
perform the integration over $[d,\infty)$. By Corollary~\ref{cor:vols},
\[
\Vol(\uDelta_s) = \sum_{i=0}^n \binom ni \sigma_{n-i} s^i
\]
for $s>d$. Therefore, 
\[
\rho_0+\rho_1 t+ \cdots +\rho_n t^n + \text{h.o.t.}
=(n+1) \int_d^\infty \sum_{i=0}^n \binom ni \sigma_{n-i} s^i \frac{t^{n+1}}{(1+st)^{n+2}}\, ds
\]
and by the same token used above, we may change the integration range back to
$[0,\infty)$:
\begin{equation}\label{eq:rhos}
\rho_0+\rho_1 t+ \cdots +\rho_n t^n + \text{h.o.t.}
=(n+1) \int_0^\infty \sum_{i=0}^n \binom ni \sigma_{n-i} s^i \frac{t^{n+1}}{(1+st)^{n+2}}\, ds
\end{equation}

The right-hand side is evaluated by using the following lemma.
\begin{lemma}
For $0\le i\le n$,
\begin{equation}\label{eq:intform}
(n+1)\int_0^\infty \frac{s^i}{(1+st)^{n+2}} ds = \frac 1{\binom ni t^{i+1}}\quad.
\end{equation}
\end{lemma}

\begin{proof}
Apply the change of variable $z=1/(1+st)$, i.e., $s=\frac 1t (\frac 1z-1)$
(treating $t$ as a positive parameter):
\[
\int_0^\infty \frac{s^i}{(1+st)^{n+2}} ds = -\int_0^1 \frac{\frac 1{t^i}\left(\frac 1z-1\right)^i}
{\frac 1{z^{n+2}}} \left(-\frac 1{tz^2}\right) dz 
= \frac 1{t^{i+1}} \int_0^1 z^{n-i}(1-z)^i dz\quad.
\]
The last integral is an instance of the {\em Beta function,\/} so we obtain
\begin{align*}
\int_0^\infty \frac{s^i}{(1+st)^{n+2}} ds &= \frac 1{t^{i+1}} B(n-i+1,i+1)
=\frac 1{t^{i+1}} \frac{\Gamma(n-i+1)\Gamma(i+1)}{\Gamma(n+2)} \\
&=\frac 1{t^{i+1}} \frac{(n-i)! i!}{(n+1)!}
\end{align*}
with the stated consequence.
\end{proof}

With this understood,~\eqref{eq:rhos} gives
\[
\rho_0+\rho_1 t+ \cdots +\rho_n t^n 
\equiv \sum_{i=0}^n \sigma_{n-i} t^{n-i} \mod (t^{n+1})\quad,
\]
and this proves the needed equality $\rho_i=\sigma_i$ for $i=0,\dots, n$.
\end{proof}


\section{From the Newton-Okounkov body to Segre zeta functions}\label{s:segre2}

\subsection{}
Theorem~\ref{thm:segre1} only gives information on the first $n+1$ coefficients of the
expansion
\[
\int_{\Delta(I)} \frac{(n+1)! t^{n+1} da_0\cdots da_n}
{(1+(a_0+\cdots + a_n)t)^{n+2}} = \sum_{i\ge 0} \rho_i t^i\quad.
\]
This begs the question of what the other coefficients may mean. 
In this section we will prove that the integral computes the `Segre zeta function' of the given 
ideal.

Let $I\subseteq \Cbb[x_0,\dots, x_n]$ be a homogeneous ideal. For $N\ge n$, let $X^{(N)}$
denote the subscheme of $\Pbb^N$ determined by the extension $I^{(N)}$ of $I$ to 
$\Cbb[x_0,\dots, x_N]$, and denote by $\iota^{(N)}$ the inclusion~$X^{(N)}\hookrightarrow \Pbb^N$.
(Thus $I=I^{(n)}$, $X=X^{(n)}$.)
In~\cite{MR3709134} it is shown that there exists a power series 
$\zeta_I(t)=\sum_{i\ge 0} s_i t^i$ such that {\em for all\/} $N\ge n$
\begin{equation}\label{eq:zeta}
\zeta_I(h)\cap [\Pbb^N] = \sum_{i= 0}^\infty s_i h^i \cap [\Pbb^N]
=\iota^{(n)}_* s(X^{(N)},\Pbb^N)\quad,
\end{equation}
where we denote by $h$ the hyperplane class in $\Pbb^N$.
Of course $h^i\cap [\Pbb^N]=0$ for $i> N$; thus, for any given $N$ only $s_0,\dots, s_N$ 
contribute nonzero components in~\eqref{eq:zeta}. It is also shown in~\cite{MR3709134} that
$\zeta_I(t)$ is rational, with poles constrained by the degree sequence of $I$.

We can let
\[
1-\zeta_I(t) = \sum_{j=0}^\infty \sigma_j t^j\quad,
\]
where for any $N\ge n$ the coefficients $\sigma_j$ are defined as in~\eqref{eq:defsig}: 
that is, the definition of $\sigma_j$ is independent of $n$ for $n\ge j$ (this follows 
from~\cite[Lemma~5.2]{MR3709134}) and we can assemble the information simultaneously
for all $j$ into a single power series.
Our last goal is to prove that the integral appearing in Theorem~\ref{thm:segre1} agrees with
this function.

\begin{theorem}\label{thm:segre2}
Let $I\subseteq \Cbb[x_0,\dots, x_n]$ be a homogeneous ideal, and let 
$\Delta(I)\subseteq \Rbb^{n+1}$
be the corresponding Newton-Okounkov body. Then
\begin{equation}\label{eq:int2}
\int_{\Delta(I)} \frac{(n+1)!\, t^{n+1} da_0\cdots da_n}
{(1+(a_0+\cdots + a_n)t)^{n+2}} = 1- \zeta_I(t)\quad.
\end{equation}
\end{theorem}

In other words, with notation as above and as in~\eqref{eq:int1}, $\rho_i = \sigma_i$ for
{\em all\/} $i\ge 0$. The parameter $t$ is implicitly assumed to be a positive real in computing
the integral. The content of Theorem~\ref{thm:segre2} is that the integral then equals a 
well-defined rational function of~$t$, and this function equals $1-\zeta_I(t)$.
In particular, the integral is independent of the choices (of a valuation and of a dehomogenizing
term) used to define $\Delta(I)$.

\subsection{}\label{ss:arguseg2}
Theorem~\ref{thm:segre1} shows that~\eqref{eq:int2} is true {\em modulo $t^{n+1}$.\/}
In particular, we have
\[
\int_{\Delta(I^{(N)})} \frac{(N+1)! t^{N+1} da_0\cdots da_N}
{(1+(a_0+\cdots + a_N)t)^{N+2}} \equiv
\int_{\Delta(I)} \frac{(n+1)! t^{n+1} da_0\cdots da_n}
{(1+(a_0+\cdots + a_n)t)^{n+2}} \mod t^{n+1}
\]
for all $N\ge n$. In order to prove Theorem~\ref{thm:segre2}, it suffices to prove that the 
two integrals are equal {\em modulo $t^{N+1}$.\/} Inductively, it suffices to show that{\small
\begin{equation}\label{eq:key}
\int_{\Delta(I^{(n+1)})} \frac{(n+2)! t^{n+2} da_0\cdots da_{n+1}}
{(1+(a_0+\cdots + a_{n+1})t)^{n+3}} \equiv
\int_{\Delta(I)} \frac{(n+1)! t^{n+1} da_0\cdots da_n}
{(1+(a_0+\cdots + a_n)t)^{n+2}} \mod t^{n+2}\quad.
\end{equation}}
Now recall that the definition of the Newton-Okounkov body $\Delta(I^{(n+1)})$ 
depends on the 
choice of a valuation, but Theorem~\ref{thm:segre1} implies that the integral modulo $t^{n+2}$ 
is independent of this choice. Thus we may assume that a valuation $v$ has been chosen
for $\Abb^n$, and our task will be to show that there exists a corresponding valuation $v'$ 
for $\Abb^{n+1}$ with respect to which~\eqref{eq:key} holds.
Theorem~\ref{thm:segre2} will follow from~\eqref{eq:key}.

\subsection{}
Let $v$ be a $\Zbb^n$-valued valuation on $\Cbb[x_1,\dots, x_n]$.
We extend $v$ to a valuation~$v'$ on $\Cbb[x_1,\dots, x_{n+1}]$ as follows, 
cf.~\cite[Definition~2.26]{MR2950767}. Given $f\in \Cbb[x_1,\dots, x_{n+1}]$, write
\[
f= x_{n+1}^e g
\]
with $x_{n+1} \not | g$. Define
\[
v'(f) = (v(g|_{x_{n+1}=0}), e)\quad.
\]
Then $v'$ is a $\Zbb^{n+1}$-valued valuation.
Given a homogeneous ideal $I$ in $\Cbb[x_0,\dots, x_n]$, denote by $I'$ the extension of
$I$ to $\Cbb[x_0,\dots, x_{n+1}]$. We want to compare the bodies $\Delta(I)\subseteq
\Rbb^{n+1}$, $\Delta(I')\subseteq \Rbb^{n+2}=\Rbb^{n+1}\times \Rbb^1$
defined by means of $v$, $v'$ respectively.
(Also recall that we are silently dehomogenizing by setting $x_0=1$; cf.~\S\ref{ss:cprels}.)

\begin{lemma}\label{lem:ext}
With notation as above,
$\Delta(I')=\Delta(I) \times \Rbb^{\ge 0}$.
\end{lemma}

\begin{proof}
The key ingredients in the construction of $\Delta(I)$, $\Delta(I')$ are the sets
\begin{align*}
U_I &:=\{(\ua, s,t)\in \Rbb^n\times \Rbb^2 \,|\, s\in \Zbb^{\ge 0}, t\in \Zbb^{\ge 0}, 
\ua\in v({I^t}_s)\} \\
U_{I'} &:=\{(\ua', s,t)\in \Rbb^{n+1}\times \Rbb^2 \,|\, s\in \Zbb^{\ge 0}, t\in \Zbb^{\ge 0}, 
\ua'\in v'({{I'}^t}_s)\}
\end{align*}
(cf.~\eqref{eq:dUI}). Here $\ua=(a_1,\dots, a_n)$, $\ua'=(a_1,\dots, a_{n+1}) = (\ua,a_{n+1})$.
Consider the map
\[
\lambda: (\Rbb^n\times \Rbb^2)\times \Rbb \to (\Rbb^{n+1}\times \Rbb^2)
\]
given by
\[
\lambda((\ua, j, t),a_{n+1})= ((\ua,a_{n+1}), a_{n+1}+j,t)
\]
The lemma will follow easily from the following statement.
\begin{claim}\label{cla:Un}
\[
U_{I'} = \lambda(U_I\times \Zbb^{\ge 0})\quad.
\]
\end{claim}

In order to verify the claim, note first that 
\begin{equation}\label{eq:dehds}
I'_j = \bigoplus_{i=0}^j x_{n+1}^{j-i} I_i \quad.
\end{equation}
This identity has a counterpart for valuations. Since $I'_j \supseteq 
x_{n+1}^{j-i} I_i$,
clearly
\begin{equation}\label{eq:vun}
v'(I'_j)\supseteq (v(I_j),0)\cup (v(I_{j-1}),1)\cup \cdots 
\cup (v(I_0),j)\quad.
\end{equation}
This union is disjoint, and for every $i$,
\[
|(v(I_{j-i}),i)| = \dim I_{j-i}
\]
by~\cite[Lemma 1.3]{MR2571958};
therefore the cardinality of the right-hand side in~\eqref{eq:vun} equals
$\sum_{i=0}^j \dim I_{j-i}$,
and this equals $\dim I'_j$ since the sum in~\eqref{eq:dehds} is direct. It follows
that 
\[
\sum_{i=0}^j |(v(I_{j-i}),i)| = |v'(I'_j)|\quad,
\]
and this shows that the inclusion in~\eqref{eq:vun} is an equality:
\begin{equation}\label{eq:valp}
v'(I'_j)=(v(I_j),0)\cup (v(I_{j-1}),1)\cup \cdots 
\cup (v(I_0),j)\quad.
\end{equation}

Now apply these considerations to powers of $I$ and of its extension $I'$.
By~\eqref{eq:valp},
\[
U_{I'} = \{(\ua', s, t)=((\ua,a_{n+1}), s, t) \in \Zbb^{n+1}\times\Zbb^2 \,|\, s\ge 0, t\ge 0,a_{n+1}\ge 0,
\ua\in v({I^t}_{s-a_{n+1}}) \}\quad.
\]

Claim~\ref{cla:Un} is simply a restatement of this identity.
\qede\smallskip

Since taking closed convex cones is preserved by linear maps, and $\lambda$ 
maps $t=1$ to $t=1$, Claim~\ref{cla:Un} implies that
\[
\uDelta(I') = \lambda|_{t=1}(\uDelta(I)\times \Rbb^{\ge 0})\quad.
\]
Lemma~\ref{lem:ext} follows by applying the isomorphisms~$\tau$ used in
Definition~\ref{def:NO}: we have the commutative diagram
\[
\xymatrix@C=50pt{
(\Rbb^n\times \Rbb^1)\times \Rbb^1 \ar[r]^-{\lambda|_{t=1}} \ar[d]_-{\tau\times\id} 
& \Rbb^{n+1}\times \Rbb^1 \ar[d]^-\tau \\
\Rbb^{n+1}\times \Rbb^1 \ar[r]^\cong & \Rbb^{n+2}
}
\]
where the bottom map is simply $((a_0,\dots, a_n),a_{n+1})\mapsto (a_0,\dots, a_{n+1})$,
and this gives the stated identification $\Delta(I)\times \Rbb^{\ge 0} = \Delta(I')$.
\end{proof}

\subsection{}
We can now complete the proof of Theorem~\ref{thm:segre2}. As we argued in~\S\ref{ss:arguseg2},
it suffices to prove~\eqref{eq:key}. We will in fact show that the two integrals appering 
in~\eqref{eq:key} agree modulo any power of $t$, if the valuation~$v'$ is used for the
construction of $\Delta(I')$ (and it will then follow that they must be equal for any choice
of valuation).

This is a straightforward consequence of Lemma~\ref{lem:ext}. Since $\Delta(I') = \Delta(I)
\times \Rbb^{\ge 0}$, 
\begin{align*}
\int_{\Delta(I')} &\frac{(n+2)!\, t^{n+2}}{(1+(a_0+\cdots +a_{n+1}) t)^{n+3}}\, 
\, d a_0\cdots da_{n+1} \\
&=\int_{\Delta(I)} \left(\int_0^\infty \frac{(n+2)!\, t^{n+2}}{(1+(a_0+\cdots +a_{n+1}) t)^{n+3}} 
da_{n+1}\right) d a_0\cdots da_n \\
&=\int_{\Delta(I)} \frac{(n+1)!\, t^{n+1}}{(1+(a_0+\cdots +a_n) t)^{n+2}} \, d a_0\cdots da_n
\end{align*}
as needed.
\qed

\begin{remark}
In the case of ideals of $\Cbb[x_0,\dots, x_n]$ generated by monomials in the variables~$x_i$,
Theorem~\ref{thm:segre2} recovers a `numerical' form of the main result of~\cite{MR3576538}.
Indeed, as seen in~\S\ref{ss:monom}, in this case $\Delta(I)$ is the ordinary Newton polytope
of $I$, and the statement of Theorem~\ref{thm:segre1} follows by setting $X_1=\cdots = X_n=t$ 
(which amounts to taking degrees) in the integral appearing in~\cite[Theorem 1.1]{MR3576538}.
As this result is independent of the number of variables, the integral in fact computes
the Segre zeta function, yielding the monomial case of~Theorem~\ref{thm:segre2}.
\qede\end{remark}

\subsection{}
To summarize, let $f_0,\dots, f_r$ be homogeneous polynomials in $\le n$ variables.
These polynomials generate an ideal $I\subseteq \Cbb[x_0,\dots, x_n]$, and we have proved 
that the integral
\[
\int_{\Delta(I)} \frac{(n+1)!\, t^{n+1}}{(1+(a_0+\cdots +a_n) t)^{n+2}} \, d a_0\cdots da_n
\]
is {\em independent of~$n$,\/} and in fact 
\[
1- \int_{\Delta(I)} \frac{(n+1)!\, t^{n+1}}{(1+(a_0+\cdots +a_n) t)^{n+2}} \, 
d a_0\cdots da_n=\zeta_I(t) 
\]
is a rational function, with various properties studied in~\cite{MR3709134}---for example,
this function can only have poles at $-1/{d_i}$, where the integers $d_i$ belong to the degree 
sequence for $I$, and its numerator is a polynomial with nonnegative coefficients. 
Some of the known properties of $\zeta_I(t)$ can probably be ascribed to properties of the
Newton-Okounkov body $\Delta(I)$. It would be interesting to explore this connection.


\end{document}